\documentclass[a4paper,twoside,11pt]{article}
\usepackage[english]{babel}
\usepackage[T1]{fontenc}
\usepackage[ansinew]{inputenc}
\usepackage{geometry}

\usepackage{lmodern}

\usepackage{graphicx}
\usepackage{amsmath} \numberwithin{equation}{section}
\usepackage{amsthm}
\usepackage{amsfonts}
\usepackage{amssymb}

\usepackage[pdfstartview=FitH]{hyperref}
\usepackage{enumerate}

\theoremstyle{plain}
\newtheorem{theorem}{Theorem}[section]

\newtheorem{corollary}[theorem]{Corollary}
 
\newtheorem{remark}[theorem]{Remark}\newtheorem{example}[theorem]{Example}
\theoremstyle{definition}
\newtheorem{definition}[theorem]{Definition}
\newtheorem{proposition}[theorem]{Proposition}    
\theoremstyle{remark}

\newcommand\C{\mathbb C}      
\newcommand\R{\mathbb R}        
\newcommand\Z{\mathbb Z}         
\newcommand\N{\mathbb N}        
\newcommand\Ha{\mathbb H}       
\newcommand\D{\mathbb D}     
\renewcommand\Im{\text{Im}}        
\renewcommand\Re{\text{Re}}

\newcommand{\hcap}{\operatorname{hcap}} 
\newcommand{\supp}{\operatorname{supp}} 
\newcommand{\til}{\mu} 
\newcommand{\meu}{\alpha}

\begin{document}
\parindent 0pt 

\setcounter{section}{0}

\title{Tightness results for infinite-slit limits of the chordal Loewner equation}
\date{\today}
\author{
Andrea del Monaco \and
Ikkei Hotta\thanks{Supported by the JSPS KAKENHI Grant no. 26800053.}\and
 Sebastian Schlei{\ss}inger\thanks{Supported by the ERC grant  ``HEVO - Holomorphic Evolution Equations'' no. 277691.}}
\maketitle

\abstract{In this note we consider a multi-slit Loewner equation with constant coefficients that describes the growth of multiple SLE  curves connecting $N$ points on $\R$ to infinity within the upper half-plane. For every $N\in\N$, this equation provides a measure valued process $t\mapsto \{\meu_{N,t}\},$ and we are interested in the limit behaviour as $N\to\infty.$ We prove tightness of the sequence $\{\meu_{N,t}\}_{N\in\N}$ under certain assumptions and address some further problems.}\\



\tableofcontents
\parindent 0pt

\section{Introduction}\label{introduction}

In \cite{MS15}, the second and third author noted that the conformal mappings for a certain multiple SLE (Schramm-Loewner evolution) process for $N$ simple curves in the upper half-plane $\Ha$ converges as $N\to\infty$.  The deterministic limit has a simple description: The conformal mappings $f_t:\Ha\to\Ha$ satisfy the Loewner PDE
$$\frac{\partial f_t(z)}{\partial t} = - \frac{\partial f_t(z)}{\partial z}\cdot M_t(z), \quad f_0(z)=z\in\Ha,$$
where $M_t$ satisfies the complex Burgers equation 
$$\frac{\partial M_t(z)}{\partial t} = -2 \frac{\partial M_t(z)}{\partial z}  \cdot M_t(z),$$
see Section \ref{Princess} for more details. In several situations, partial differential equations of this type appear to describe the limit of $N$-particle systems; see  \cite{MR1176727, MR1217451, MR1440140}.\\

In Section \ref{multipleSLE}, we consider again the same multiple SLE measure for $N$ curves connecting $N$ points on $\R$ with $\infty.$ We describe the growth of these curves by a Loewner equation with weights that correspond to the speed for these curves in the growth process, and we obtain an abstract differential equation for limit points as $N\to\infty$ (Corollary \ref{thm:2}).\\
Furthermore, in Section \ref{Confucius} we see that an equation of a similar type also appears in the limit behaviour of a Loewner equation describing the growth of trajectories of a certain quadratic differential.

\section{Tightness of a multiple SLE process}\label{multipleSLE}

\subsection{Geometry and Loewner Theory}
In this section we briefly recall the general background of hulls in the upper half-plane and the chordal Loewner equation.\\

A domain $D\subsetneq\hat{\C}$ is said to be a \emph{Jordan domain} if $\partial D$ is homeomorphically equivalent to the unit circle $\mathbb{T}=\partial\D$.
Let $\Gamma$ be a subset of $\overline{D}$ such that there exist some $T>0$ and a homeomorphism $\gamma\colon[0,T]\longrightarrow\Gamma$ with $\gamma(0,T)\subset D$ and $\gamma(0)\in\partial D$.
Then, if $\gamma(T)\in D$, the set $\Gamma \cap D = \Gamma \setminus \gamma(0)$ is said to be a \emph{slit} in $D$, and if $\gamma(T)\in \partial D$ as well, $\Gamma$ is referred to as a \emph{chord} (in $D$).

Since by the Riemann Mapping Theorem (see, \textit{e.g.},~\cite[Section 1.1]{Pom:1975}) $D$ is conformally equivalent to the upper half-plane $\Ha = \{z\in\C \mid \Im(z)>0\}$, it suffices to consider the case $D=\Ha$ and $\gamma(0)\in\R$. 
In this setting, in particular, one may also introduce the more general notion of \emph{hull}, i.e. a subset $A\subset\Ha$ such that $\overline{A}\cap\Ha= A$ and $\Ha\setminus A$ is simply connected.\\
It is known that if $A\subset\Ha$ is a bounded hull, then there exists a unique conformal mapping $g_A:\Ha\setminus A\to\Ha$ with \emph{hydrodynamic normalization} (see \cite[Proposition 3.34]{Lawler:2005}), meaning that
$$g_A(z) = z+\frac{b}{z}+\tilde{g}(z) \quad \text{as $z\to\infty$}$$ 
for a holomorphic function $\tilde{g}$ with $\angle\lim_{z\to \infty}z\cdot\tilde{g}(z)=0$.\\
The quantity $b=\hcap(A)\geq0$ is called the \emph{half-plane capacity} of $A$.\\
The mapping $g_A$ can be embedded into the solution of a Loewner equation as follows. Let $T>0$ be defined by $2T=\hcap(A)$. Then there exists a family $\{\mu_t\}_{t\in[0,T]}$ of probability measures on $\R$, with the property that $t\mapsto \int_\R\frac{1}{z-u}\,\mu_t(du)$ is measurable for every $z\in\Ha$, such that the solution $\{g_t\}_{t\in[0,T]}$  of the chordal Loewner equation 
\begin{equation}\label{loop}
\left\lbrace
\begin{aligned}
&\frac{dg_t(z)}{dt} = \int_\R \frac{2}{g_t(z) - u}\,\mu_t(du) \quad \text{for almost every $t\in[0, T]$}\\
&g_0(z)=z\in\Ha
\end{aligned}
\right.
\end{equation} satisfies $g_A = g_T.$
This follows from \cite[Theorem 5]{MR1201130} and considering the time-reversed flow and the inverse mapping $g_A^{-1}$.\\

Conversely, one can always solve \eqref{loop} and obtain conformal mappings with hydrodynamic normalization; see \cite[Theorem 4]{MR1201130} or \cite[Theorem 4.5]{Lawler:2005}.\\
For $z\in\Ha$ fixed, the solution $t\mapsto g_t(z)$ of \eqref{loop} may have a finite lifetime $T(z)>0$, namely $g_t(z)\in\Ha$ for all $t<T(z)$ and $\Im(g_t(z))\to0$ as $t\uparrow T(z)$.\\
If we fix a time $t>0$ and let $K_t=\{z\in\Ha \,|\, T(z)\leq t\},$ then $K_t$ is a (not necessarily bounded) hull  and the mapping $z\mapsto g_t(z)$  is the conformal mapping from $\Ha\setminus K_t$ onto $\Ha$ with hydrodynamic normalization. Furthermore, the hulls $K_t$ are strictly growing, i.e. $K_s\subsetneq K_t$ whenever $s < t,$ and $\hcap(K_t)=2t.$\\

When the hull $A$ is a slit $\Gamma$, equation \eqref{loop} necessarily has the form
\begin{equation}\label{loop2}
\frac{dg_t(z)}{dt} =  \frac{2}{g_t(z) - U(t)}, \quad g_0(z)=z\in\Ha,
\end{equation} with a unique, continuous driving function $U:[0,T]\to \R$ (see \cite{GM13}, and the references therein, for more details). In this case, we  obtain a parametrization $\gamma$ of $\Gamma$ by setting $\gamma(0,t]=K_t$, which is equivalent to requiring $\hcap(\gamma(0,t])=2t$. We call $\gamma$ the \emph{parametrization by half-plane capacity} of $\Gamma$.\\

More generally, if $A$ is the union of $n$ slits $\Gamma_1,...,\Gamma_n$ with pairwise disjoint closures, i.e. $\overline{\Gamma_j} \cap \overline{\Gamma_k}=\emptyset$ whenever $j\not=k,$ then \eqref{loop} must have the form
\begin{equation}\label{loop3}
\frac{dg_t(z)}{dt} =  \sum_{j=1}^n \frac{2 \lambda_j(t)}{g_t(z) - U_j(t)}, \quad g_0(z)=z\in\Ha,
\end{equation}
where $U_j:[0,T]\to \R$ are continuous and $\lambda_j:[0,T]\to [0,1]$ are measurable functions with $\sum_{j=1}^n \lambda_j(t)=1$ for every $t;$ see \cite[Theorem 2.54]{Boehm}. In this way, we obtain parametrizations  $\gamma_1,...,\gamma_n$ of $\Gamma_1,...\Gamma_n$ by requiring $K_t= \cup_{j=1}^n \gamma_j(0,t].$\\
It is worth noting that, for $n>1$, a representation of $A$ by \eqref{loop3} is not unique. For example, we could first generate slit $\Gamma_k$ only, i.e. $\lambda_k(t)=1=1-\lambda_{j}(t)$ for $j\not=k$ and $t$ small enough.
 \begin{remark}
The coefficients $\lambda_j(t)$ can be thought of as the speed of growth of the slit $\Gamma_{j}$ at time $t$. More precisely, we have the following relation: \\
Fix $j$ and $t_0\geq0,$ assume that $g_t$ is differentiable at $t=t_0$ and consider the curve $\tilde{\gamma}(h)=g_{t_0}(\gamma_{j}[t_0,t_0+h]).$ Let $b(h):=\hcap(\tilde{\gamma}(0,h])$ be the half-plane capacity of the slit $\tilde{\gamma}(0,h]$. Then $b(h)$ is differentiable at $h=0$ with $b'(0) = \lambda_{j}(t_0),$ see \cite[Theorem 2.36]{Boehm}.\end{remark}


\subsection{Single and Multiple SLE}
In what follows, $\kappa\in(0,4]$ is a fixed parameter and $D\subsetneq\C$ is a Jordan domain.

Fix two points $x,y \in \partial D$ and assume that $\partial D$ is analytic in neighbourhoods of $x$ and $y$.\\
The chordal Schramm-Loewner evolution (SLE) of a random curve $\Gamma\subset D$ for the data $D,x,y,\kappa$ can be viewed as a certain probability measure $\mu_{D,\kappa}(x,y)$ on the space of all chords connecting the points $x$ and $y$ within $D.$ As one property of SLE is conformal invariance, it suffices to describe the SLE when $D=\Ha$, $x=0$, and $y=\infty$.
 In this setting, the evolution of $\Gamma$ can be described efficiently as follows.  Let $\gamma$ be a parametrization of $\Gamma$ with $\gamma(0)=0$ and assume that $\gamma[0,T]$ is parametrized by half-plane capacity for every $T>0$. The random conformal mapping $g_t:=g_{\gamma(0,t]}$ then satisfies the Loewner equation
\begin{equation}\label{loewner1}
\frac{dg_t(z)}{dt} = \frac{2}{g_t(z) - \sqrt{\kappa}B_t}, \quad g_0(z)=z,
\end{equation}
where $B_t$ is a standard one-dimensional Brownian motion.\\
Notice that one may also consider SLE for $\kappa>4$. But then the measure is no longer supported on simple curves, and we are not interested in such a case here. For further information and a thorough treatment of SLE we refer to~\cite{Lawler:2005}.\\

Next, we describe multiple SLE as it was introduced in \cite{MR2310306}.\\

Let $N\in\N$ and fix $2N$ pairwise distinct points $p_1,...,p_{2N}\in \partial D$ in counter-clockwise order. Assume that $\partial D$ is analytic in a neighbourhood of $p_k$, $k=1,\ldots,2N$.\\
We call the pair $(\mathbf{x}, \mathbf{y})$ of two tuples $\mathbf{x}=(x_1,...,x_N),\mathbf{y}=(y_1,...,y_N)$ a \emph{configuration} for these points if 
\begin{itemize}
\item[a)] $\{x_1,...,x_N,y_1,...,y_N\}=\{p_1,...,p_{2N}\}$,
\item[b)] there exist $N$ pairwise disjoint chords $\gamma_k$ connecting $x_k$ to $y_k$ within $D$, $k=1,\ldots,N$,
\item[c)]  $x_1=p_1$ and $x_1,x_2,...,x_N,$ as well as $x_1,x_k,y_k$, for every $k\geq 2,$ are in counter-clockwise order.

\end{itemize}
The points in $\mathbf{x}$ can be thought of as the starting points of these chords. Then $\mathbf{y}$ represents the end points and the assumption in c) just prevents us from getting a new configuration by exchanging a starting point of one curve with its endpoint.
A simple combinatorial calculation gives that there exist 
$$C_N = \frac{(2N)!}{(N+1)!\,N!}$$

many configurations for $2N$ points.\\ 

Fix now a configuration $(\mathbf{x},\mathbf{y})$.
 The \emph{configurational multiple SLE} $Q_{D,\kappa}(\mathbf{x},\mathbf{y})$ is a positive finite measure on the space of all $N-$tuples $(\gamma_1,...,\gamma_N)$, where $\gamma_k$ is a  chord in $D$ connecting $x_k$ and $y_k$ and $\gamma_k\cap\gamma_j=\emptyset$ whenever $j\not=k$.  One may contruct the $Q_{D,\kappa}(\mathbf{x},\mathbf{y})$ by means of the Brownian loop measure (see \cite{MR2310306} for details).\\
If we let $H_{D,\kappa}(\mathbf{x},\mathbf{y})$ be the mass of $Q_{D,\kappa}(\mathbf{x},\mathbf{y}),$ then we can write $$Q_{D,\kappa}(\mathbf{x},\mathbf{y})=H_{D,\kappa}(\mathbf{x},\mathbf{y})\cdot \mu_{D,\kappa}(\mathbf{x},\mathbf{y})$$  for some probability measure $\mu_{D,\kappa}(\mathbf{x},\mathbf{y})$.\\

 Thus, one may view $Q_{D,\kappa}(\mathbf{x},\mathbf{y})$ as a probability measure for the underlying configuration with weight $H_{D,\kappa}(\mathbf{x},\mathbf{y})$. Then we may use such weights as partition functions to combine multiple SLE for different configurations. Namely, if $\mathbf{p}=(p_1,...,p_{2N})$ and $S(\mathbf{p})$ is the set of all configurations, then the probability for $(\mathbf{x},\mathbf{y})\in S(\mathbf{p})$ will be given by
\begin{equation}\label{QueenMom}
p(\mathbf{x},\mathbf{y})=\frac{H_{D,\kappa}(\mathbf{x},\mathbf{y})}{\sum_{(\mathbf{v},\mathbf{w})\in S(\mathbf{p})} H_{D,\kappa}(\mathbf{v},\mathbf{w})}.
\end{equation}

 \begin{example} Consider the case $N=2$ and $\kappa=3$. Then there are two possible configurations $C_1$ and $C_2$, and $\mu_{D,3}(\{C_1,C_2\})$ describes the scaling limit for the Ising model with corresponding boundary conditions (see~\cite{MR2515494}). The probability $p$ for obtaining configuration $C_1$ is given by $$p=\frac{H_{D,3}(C_1)}{H_{D,3}(C_1)+H_{D,3}(C_2)}.$$
\end{example}

On the other hand, $H_{D,\kappa}(\mathbf{x},\mathbf{y})$ may also be used to write down a Loewner equation that governs the growth of multiple SLE curves, see \cite[Section 4]{MR2310306}.

Again, because of conformal invariance, it suffices just to consider the case $D=\Ha$, where $p_1,...,p_{2N}\in \R\cup{\{\infty\}}$. In this setting, the number $H_{\Ha, \kappa}(\mathbf{x}, \mathbf{y})$ is known explicitly only for some special cases:
\begin{itemize}
\item[(i)] for $N=1$ and $(x,y)=(0,\infty)$, one simply takes $H_{\Ha, \kappa}(0,\infty)=1$ as a definition, which would then yield $Q_{D,\kappa}=\mu_{\Ha,\kappa}$, i.e. the chordal SLE probability measure as described in~\ref{loewner1};
\item[(ii)] if $N=1$ and $x,y\in\R$, then $H_{\Ha,\kappa}(x,y)=|y-x|^{-2b},$ $b=\frac{6-\kappa}{2\kappa}$;
\item[(iii)] a special case for $\kappa=2$ is given in \cite{MR2310306} (see Remark after Proposition 3.3);
\item[(iv)] for $N=2$, $H_{\Ha,\kappa}\big((x_1,x_2),(y_1,y_2)\big)$ can be expressed by a formula involving hypergeometric functions (see \cite[Proposition 3.4]{MR2310306}).
\end{itemize}
 
We point out that multiple SLE can also be approached by requiring certain properties for the multi-slit Loewner equation, which leads to local properties of $H_{D, \kappa}(\mathbf{x}, \mathbf{y})$ as a partition function. A framework for describing $H_{D, \kappa}(\mathbf{x}, \mathbf{y})$ as the solution to certain differential equations is discussed in the recent works~\cite{MR3294954, MR3294955, MR3296159, MR3296160, puregeo}. We also refer to the articles \cite{MR2004294, MR2187598, Graham2007, MR2358649}.

\begin{remark}
Notice that one may consider $Q_{\Ha,\kappa}(\mathbf{x}, \mathbf{y})$ also for a configuration where $y_j=y_k$ (or $x_j=x_k$,  
or both) for certain $j\not=k$. This is done by considering the disjoint case $y_j\not=y_k$ first and then taking a scaled limit.
\end{remark}

In particular, if $(x_1,...,x_N) = (\infty,...,\infty)$, then one has
\begin{equation}\label{infinity1}
H_{\Ha,\kappa}((x_1,...,x_N), \infty) := H_{\Ha,\kappa}((x_1,...,x_N), (\infty,...,\infty)) :=\displaystyle \prod_{1\leq j<k\leq N}(x_k-x_j)^{2/\kappa}.
\end{equation}
See~\cite[Section 4.6]{MR2187598}, and the references therein, for more details.

\subsection{The chordal Loewner equation for \texorpdfstring{$H_{\Ha,\kappa}((x_1,...,x_N), \infty)$}{}}
Let $N\in\N$ and $x_{N,1}<...<x_{N,N}$ be $N$ points on $\R$. The growth of $N$ random curves from $\mu_{\Ha,\kappa}((x_{N,1},...,x_{N,N}), \infty)$ can be described by a Loewner equation as follows:\\

First, choose $\lambda_{N,1},...,\lambda_{N,N}\in(0,1)$ such that $\sum_{k=1}^N\lambda_{N,k}=1.$\\

Next, we define $N$ random processes $V_{N,1},...,V_{N,N}$ on $\R$ as the solution of the SDE system
\begin{equation}\label{sigma}
 dV_{N,k}(t) = \sum_{j\not=k}\frac{2(\lambda_{N,k}+\lambda_{N,j})}{V_{N,k}(t)-V_{N,j}(t)}dt+\sqrt{\kappa \lambda_{N,k}}dB_{N,k}(t), \quad V_{N,k}(0)=x_{N,k},
\end{equation}
where $B_{N,1},...,B_{N,N}$ are $N$ independent standard Brownian motions and $\kappa\in[0,4]$. 
 Although multiple SLE was only defined for $\kappa\in(0,4]$, in this particular case one may also consider the deterministic case $\kappa=0$.

The corresponding $N$-slit Loewner equation
\begin{equation}\label{multi2}
\frac{d}{dt}g_{N,t}(z) = \sum_{k=1}^N\frac{2\lambda_{N,k}}{g_{N,t}(z)-V_{N,k}(t)}, \quad g_{N,0}(z)=z\in\Ha,
\end{equation}
describes the growth of $N$ multiple SLE curves growing from $x_{N,1},...,x_{N,N}$ to $\infty;$ see \cite{MR2187598}, p. 1130 (where the function $Z$ is the partition function \eqref{infinity1}, see equation (4) on p. 1138). The function $z\mapsto g_{N,t}(z)$ is the conformal  mapping $g_{\gamma_{N,1}[0,t]\cup ...\cup \gamma_{N,N}[0,t]}$ for $N$ random simple curves $\gamma_{N,k}:[0,\infty)\to \overline{\Ha}$, which are non-intersecting and $\gamma_{N,k}(0)=x_{N,k}$.\\

We are interested in the limit $N\to\infty$ of the growing curves, i.e. the convergence of $\gamma_{N,1}[0,t]\cup ...\cup \gamma_{N,N}[0,t]$ to a  hull $K_t$. To be more precise, we would like to answer the following question once that some $t>0$ has been fixed: under which conditions does the sequence $\Ha\setminus (\gamma_{N,1}[0,t]\cup ...\cup \gamma_{N,N}[0,t])$ of domains converge to a (simply connected) domain $\Ha\setminus K_t$ with respect to kernel convergence (check Figure~\ref{fig1})?\\
According to Carath\'{e}odory's Kernel Theorem (Theorem 1.8 in~\cite{Pom:1975}), the above question is equivalent to asking for locally uniform convergence of the mappings $g_{N,t}$ to a conformal mapping $g_t:\Ha\setminus K_t\to\Ha.$ Also, we would like to be able to describe $g_t$ again by a Loewner equation.\\

Let $\delta_x$ be the point measure in $x$ with mass 1 and  define
\begin{equation}\label{babygeorge}
\meu_{N,t} \, = \, \sum_{k=1}^N \lambda_{N,k} \delta_{V_{N,k}(t)}.
\end{equation}

Then equation \eqref{multi2} can be written as 
\begin{equation}\label{Poma}
\frac{d}{dt}g_{N,t}=\int_{\R}\frac{2}{g_{N,t}-u}\, d\meu_{N,t}(u).
\end{equation}

Assume now that
\begin{equation}\label{assum}
\meu_{N,0} \overset{\mathbf w}{\longrightarrow} \meu \quad \text{as} \quad N\to\infty \;,
\end{equation}
where we denoted with ``$\overset{\mathbf w}{\longrightarrow}$'' the weak convergence and where $\alpha$ is again a probability measure. We wish to know whether the sequence $\{\meu_{N,t}\}_{N\in\N}$ of stochastic measure-valued processes converges. \\
In what follows, we show that, under certain assumptions for $x_{N,k}$ and $\lambda_{N,k}$, this sequence is tight and that each limit process satisfies the same differential equation. 

\begin{figure}[ht]
\rule{0pt}{0pt}
\centering
\includegraphics[width=6.5cm]{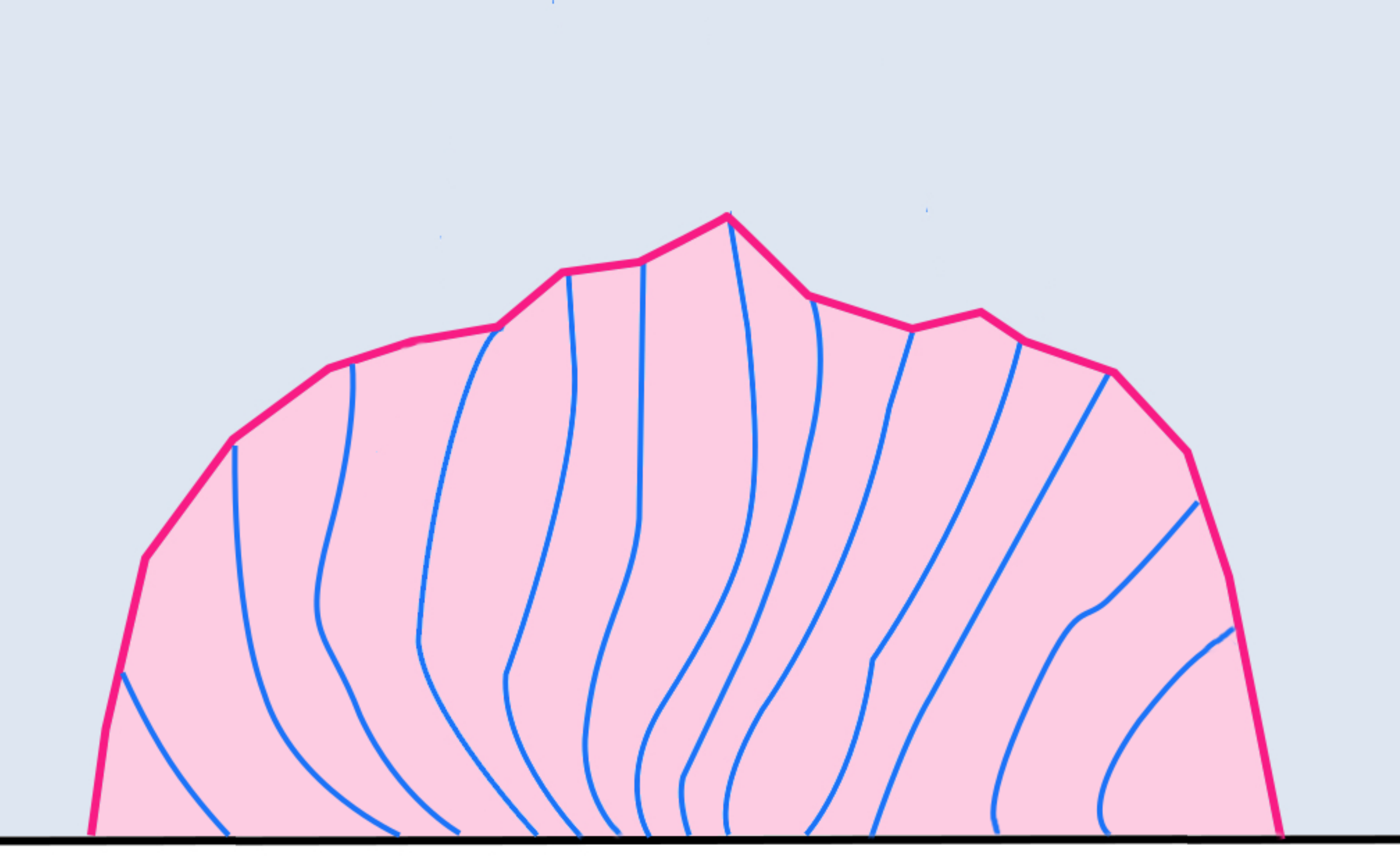}
\caption{Kernel convergence of the complement of slits.}
\label{fig1}
\end{figure}

\begin{definition}
Fix $T>0$ and let $\mathcal{P}(\R)$ be the space of probability measures on $\R$ endowed with the topology of weak convergence
(which is a metric space due to the well-known L\'{e}vy-Prokhorov metric). We denote by $\mathcal{M}(T)=C([0,T], \mathcal{P}(\R))$ the space of all continuous measure-valued processes on $[0,T]$ endowed with the topology of uniform convergence.
\end{definition}
For every $N\in\N,$ $\meu_{N,t}$ can be regarded as a random element from $\mathcal{M}(T)$.

\subsection{Tightness}
 We call a sequence $\{\mu_N\}_{N\in\N}$ of random elements from $C([0,T], \R)$ (or $\mathcal{M}(T)$) \emph{tight} if there exists a subsequence which converges in distribution. By Prohorov's Theorem (\cite[Section 5]{MR1700749}), this coincides with the usual definition of tightness.\\
We are now going to list certain conditions that guarantee tightness of the sequence $\{\meu_{N,t}\}_{N\in\N}$  defined in \eqref{babygeorge}. 

First of all, we make the following assumption:
\begin{equation}\label{max}\tag{a}
\text{there exists} \; C>0 \; \text{such that for every} \; N\in\N \; \text{it holds} \; \max_{k\in\{1,...,N\}} \lambda_{N,k} \leq \frac{C}{N}.
 \end{equation}

Now,  we introduce the ``empirical distribution'' $$\til_{N,t}=\sum_{k=1}^N \frac{1}{N} \delta_{V_{N,k}(t)}$$ and we let $L_N:[0,1]\to[0,1]$ be defined as $L_N(k/N) = \sum_{j=1}^k \lambda_{N,j}$ for $k=0,\ldots,N$. Next, we extend $L_N$ to the entire unit interval $[0,1]$ by linear interpolation.
Then the family $\{L_N\}_{N\in\N}$ is uniformly bounded by $1$ and equicontinuous by \eqref{max}. The Ascoli--Arzel\`{a} Theorem implies that it is precompact. We will hence assume that the limit exists:
\begin{equation}\label{imp:2}
L_N(x) \to L(x) \; \text{uniformly on} \; [0,1] \; \text{as} \; N\to\infty.\tag{b}
 \end{equation}

Notice that if $F_{N,t}(x)=\meu_{N,t}(-\infty,x]$ and $G_{N,t}(x)=\til_{N,t}(-\infty, x]$ are the cumulative distribution functions, we have that
\begin{equation}
\label{imp:1}
F_{N,t}(x) = L_N(G_{N,t}(x)).
\end{equation}

Finally, the last assumption is rather a technical condition. Namely, we assume that $\til_{N,0}$ converges weakly to a probability measure $\til$ in such a way that there exists a $C^2$-function $\varphi:\R\to [1,\infty)$, with $\varphi', \varphi''$ bounded and $\varphi(x)\to\infty$ for $x\to\pm\infty$ such that
\begin{equation}\label{cond1}
\sup_{N\in\N} \int_\R \varphi(x)d\til_{N,0}(x)< +\infty.\tag{c}
\end{equation}

Let $C^2_b(\R,\C)$ be the space of all twice continuously differentiable functions $f:\R\to \C$  such that $f, f', f''$ are bounded.%

\begin{theorem}\label{thm:1}
Let $T>0$. Then, under the assumptions \eqref{max}, \eqref{imp:2} and \eqref{cond1}, the sequences $\{\til_{N,t}\}_N$ and $\{\meu_{N,t}\}_N$ are tight with respect to $\mathcal{M}(T)$.\\
Moreover, if $\til_{N_k,t}$ is a converging subsequence of $\{\til_{N,t}\}_N$ with limit $\til_t$, then
\begin{enumerate}
\item{$\alpha_{N_k,t}$ converges to the process $\alpha_t$, and for every $t\in[0,T]$ the cumulative distribution function $F_t$ of $\alpha_t$ is given by
\begin{equation}\label{TheQueen}
F_t(x)=L\circ G_t(x)
\end{equation}
where $G_t$ is the cumulative distribution function of $\til_t$;}
\item{$\til_t$ satisfies the (distributional) differential equation
\begin{equation}\label{mck:1}
\left\lbrace
\begin{aligned}
&\frac{d}{dt}\left(\int_\R f(x)\, d\til_t(x)\right) =  2\int_{\R^2} \frac{f'(x)-f'(y)}{x-y} \,d\til_t(x) d\meu_t(y)\\
&\til_0=\til
\end{aligned}
\right.
\end{equation} 
for all $f\in C^2_b(\R, \C)$.}
\end{enumerate}
\end{theorem}%

\begin{remark}
The conditions \eqref{imp:2} and \eqref{cond1} are natural in the sense that we should assume convergence of the initial conditions $x_{N,k}$ and the coefficients $\lambda_{N,k}$, which are encoded in the functions $L_N.$ If some $\lambda_{N,k}$ do not converge to $0$ as $N\to\infty,$ then some part of the measure $\meu_{N,t}$ may escape to infinity as $N\to\infty$, see Example \ref{Johnny}.
\end{remark}
\begin{proof} 
To begin with, we notice that proving tightness of $\{\til_{N,t}\}_N$ can be reduced to proving  tightness of stochastic \emph{real-valued} processes (see \cite[Section 3]{MR1217451} and also \cite[Section 1.3]{MR968996}).\\
Thus, the sequence $\{\til_{N,t}\}_N$ is tight if 
$$\left\{\int_\R \varphi(x) d\til_{N,t}(x)\right\}_{N} \quad \text{and} \quad \left\{\int_\R f(x) d\til_{N,t}(x)\right\}_{N}$$ are tight sequences (with respect to the space $C([0,T], \R)$ with uniform convergence) for all $f\in C^2_b(\R,\C).$\\ 
Now, let $f\in C^2_b(\R,\C)$; It\={o}'s formula gives

\begin{eqnarray*}
 && d\left(\int_\R f(x) \, d\til_{N,t}(x)\right) = d\left(\sum_{k=1}^N \frac{1}{N} f(V_{N,k}(t)) \right) \\
&=& \sum_{k=1}^N  \frac{1}{N}\left( f'(V_{N,k}(t))\sum_{j\not=k}\frac{2(\lambda_{N,k}+\lambda_{N,j})}{V_{N,k}(t)-V_{N,j}(t)}\,dt + 
\frac{\kappa \lambda_{N,k}}{2} f''(V_{N,k}(t)) \,dt +  f'(V_{N,k}(t))\sqrt{\kappa \lambda_{N,k}}\,dB_{N,k}(t) \right) \\
&=& \sum_{k=1}^N  \frac{1}{N} f'(V_{N,k}(t))\sum_{j\not=k}\frac{2\lambda_{N,j}}{V_{N,k}(t)-V_{N,j}(t)} \,dt  + 
\sum_{k=1}^N  \frac{\lambda_{N,k}}{N} f'(V_{N,k}(t))\sum_{j\not=k}\frac{2}{V_{N,k}(t)-V_{N,j}(t)} \,dt + \dots  \\
&=& \int_{}\int_{x\not=y} \frac{2f'(x)}{x-y} \,
d\til_{N,t}(x) d\meu_{N,t}(y) \,dt+ \int_{}\int_{x\not=y} \frac{2f'(x)}{x-y} \,
d\meu_{N,t}(x) d\til_{N,t}(y) \, dt + \dots\\
&=& 2\int_{}\int_{x\not=y} \frac{f'(x)-f'(y)}{x-y} \,
d\til_{N,t}(x) d\meu_{N,t}(y)  \,dt + \ldots\\
&=& \underbrace{2\int_{}\int_{\R^2} \frac{f'(x)-f'(y)}{x-y} \,
d\til_{N,t}(x) d\meu_{N,t}(y)}_{=:A_N(t)}  \,dt +  \underbrace{2\sum_{k=1}^N \frac{\lambda_{N,k}}{N}f''(V_{N,k}(t))}_{=:B_N(t)} \,dt \\
&+&\underbrace{\frac{\kappa}{2} \sum_{k=1}^N \frac{\lambda_{N,k}}{N} f''(V_{N,k}(t))}_{=:C_N(t)} \,dt +  \underbrace{\sum_{k=1}^N  \frac1{N} f'(V_k(t))\sqrt{\kappa \lambda_{N,k}}}_{=:D_N(t)}\,dB_{N,k}(t).
\end{eqnarray*} 

As $f'$ and $f''$ are bounded and $\lambda_{N,k}\leq C/N$, it is easy to see that $A_N(t)$ is uniformly bounded and that $B_N(t), C_N(t), D_N(t)$ all converge to $0$ as $N\to \infty.$  By the ``stochastic Ascoli--Arzel\`{a} Theorem'' (\cite[Thm. 7.3]{MR1700749}), we conclude that $\left\{\int_\R f(x) d\til_{N,t}(x)\right\}_{N\in\N}$ is tight. Plus, in view of the boundedness of both $\varphi'$ and $\varphi''$, thanks to assumption \eqref{cond1}, the same reasoning also implies tightness of the sequence $\left\{\int_\R \varphi(x) d\til_{N,t}(x)\right\}_{N\in\N}.$ Hence, $\til_{N,t}$ is tight and each limit process satisfies equation \eqref{mck:1}. \\
Finally, it follows from \eqref{imp:1} and assumption \eqref{imp:2} that the subsequence $\meu_{N_k,t}$ converges provided the convergence of $\til_{N_k,t}$. In particular, it follows that relation \eqref{TheQueen} holds for the limit processes.
\end{proof}
Now we can easily show that if $\til_{N_k,t}$ is a converging subsequence, then $g_{N_k,t}$ converges as well.\\
 
First, let $\mathcal{C}$ be the set of all $M(z)=\int_\R \frac{2}{z-u} d\beta(u)$, where $\beta$ is a probability measure. So $M$ is 2 times the Cauchy transform (or Stieltjes transform) of $\beta.$ The measure $\beta$ can be recovered from $M$ by the Stieltjes-Perron inversion formula (see \cite[Theorem F.2]{MR2953553}). Denote its distribution function by $F(x).$ Then $L\circ F(x)$ is also a distribution function,  which corresponds to a measure $\hat{\beta}.$ In this way, we obtain a map $\mathcal{L}:\mathcal{C}\to \mathcal{C}$ defined as
$$\int_\R \frac{2}{z-u} d\beta(u) \mapsto \int_\R \frac{2}{z-u} d\hat{\beta}(u).$$

The limit of the Loewner equation can now be described as follows.
\begin{corollary}\label{thm:2}
Let $\til_{N_k,t}$ be a converging subsequence with limit $\mu_t$. Then $g_{N_k,t}$ converges in distribution with respect to locally uniform convergence to $g_t,$ the solution of the Loewner equation
\begin{equation}\label{Loe:1}
\frac{d}{dt} g_t = (\mathcal{L} \circ M_t)(g_t),
\end{equation}
where $M_t=\int_\R \frac{2}{z-u}d\mu_t(u)$ solves the (abstract) differential equation 
\begin{equation}\label{abstractBurgers}
\left\lbrace
\begin{aligned}
&\frac{\partial}{\partial t}M_t =  - \frac{\partial}{\partial z}M_t \cdot (\mathcal{L} \circ M_t) -M_t  \cdot \frac{\partial}{\partial z}(\mathcal{L} \circ M_t),\\
&M_0(z) = \int_\R \frac{2}{z-u} d\til(u).
\end{aligned}
\right.\quad.
\end{equation}

\end{corollary}

\begin{remark}\label{Napoli} The convergence of $\meu_{N,t}$ and $g_{N,t}$ would follow immediately if we knew that equation \eqref{abstractBurgers} (or, equivalently, \eqref{mck:1}) had a unique solution. If $\lambda_{N,k}=\frac1{N}$, then \eqref{abstractBurgers} is a usual PDE and uniqueness can be shown easily (see Section \ref{Princess}).
\end{remark}

In order to prove the above corollary, we will need the following control-theoretic result.

\begin{theorem}\label{control} Fix some $t>0$. Let $\lambda$ be the Lebesgue measure on $[0,t]$ and let $\mathcal{N}(t)$ be the space of all finite measures on $\R\times [0,t]$ endowed with the topology of weak convergence. \\
Let $\{\beta_{N,s}\}_{N\in \N}$ be a sequence of processes from $\mathcal{M}(t)$ and assume  $\beta_{N,s} \times \lambda \in \mathcal{N}(t)$ converges weakly to $\beta_s \times \lambda \in \mathcal{N}(t)$ as $N\to\infty$. Denote with $h_{N,s},$ $s\in[0,t]$, the solution to the Loewner equation
$$ \frac{d}{ds} h_{N,s}(z) = \int_\R \frac{2}{h_{N,s}(z)-u} \, d\beta_{N,s}(u), \quad h_{N,0}(z)=z.$$
Then $h_{N,t}$ converges locally uniformly to $h_t,$ where $h_s$, $s\in[0,t],$ is the solution to
$$ \frac{d}{ds} h_s(z) = \int_\R \frac{2}{h_s(z)-u} \, d\beta_s(u), \quad h_0(z)=z.$$
\end{theorem}
 
A proof of the above theorem can be found in \cite[Proposition 1]{MR2919205} or \cite[Theorem 1.1]{quantum}. Notice that even though both results consider the radial Loewner equation, the proofs can be easily adapted to the chordal case.

\begin{proof}[Proof of Corollary \ref{thm:2}]
For $z\in\Ha$, let $f(x)=\frac{2}{z-x}.$ Then $f\in C^2_b(\R,\C).$ Define now $M_t(z)=\int_\R f(x) \,d\til_t(x)$; then $(\mathcal{L} \circ M_t)(z)=\int_\R f(x) \,d\meu_t(x)$, where $\meu_t$ is the limit of $\meu_{N_k,t},$ and Theorem \ref{thm:1} implies
\begin{eqnarray*}
\frac{\partial }{\partial t}M_t(z) &=&  4\int_{\R^2} \frac{\frac{1}{(z-x)^2}-\frac1{(z-y)^2}}{x-y} \,d\til_t(x) d\meu_t(y) 
=  4\int_{\R^2} \frac{2z-x-y}{(z-x)^2(z-y)^2} \,d\til_t(x) d\meu_t(y)\\
&=&\int_{\R^2} \frac{2}{(z-x)^2}  \frac{2}{(z-y)} +  \frac{2}{(z-x)} \frac{2}{(z-y)^2} \,d\til_t(x) d\meu_t(y)\\
&=&  - \frac{\partial}{\partial z}M_t \cdot (\mathcal{L} \circ M_t) -M_t  \cdot \frac{\partial}{\partial z}(\mathcal{L} \circ M_t).
\end{eqnarray*}

Furthermore, let $g_t$ be the solution to 
$$\frac{d}{dt}g_t = (\mathcal{L}\circ M_t)(g_t), \quad g_0(z)=z.$$
Fix some $t> 0$. 
The canonical mapping $\mathcal{M}(t) \ni \meu_s \mapsto \meu_s \times \lambda \in \mathcal{N}(t)$ is continuous. It follows from the Continuous Mapping Theorem (see \cite{MR1700749}, p. 20) that $\meu_{N_k,s} \times \lambda$ converges in distribution with respect to weak convergence to $\meu_s \times \lambda$. \\ 
Hence, Theorem \ref{control} and again the Continuous Mapping Theorem imply that $g_{N_k,t}$, which is the solution to \eqref{Poma},  converges in distribution to $g_t$ with respect to locally uniform convergence.
\end{proof}

\subsection{The simultaneous case}\label{Princess}

In the case $\lambda_{N,k} = \frac1{N}$ for all $k$, which we call the \emph{simultaneous} case, equation \eqref{sigma} becomes
\begin{equation}\label{sigma_simul}
 dV_{N,k} = \sum_{j\not=k}\frac{4/N}{V_{N,k}-V_{N,j}}dt+\sqrt{\kappa/N}dB_{N,k},
\end{equation}
a process that is quite similar to a Dyson Brownian motion.\\
Note that in such a case $\til_{N,t}=\meu_{N,t}$ and $\mathcal{L}$ is the identity map. If $\meu_t$ is the limit of a converging subsequence of $\{\meu_{N,t}\}_{N}$ and $M_t(z) = \int_\R \frac{2}{z-u}d\meu_t(u),$ then $M_t$ satisfies the complex Burgers equation  
\begin{equation}\label{burgers}
\left\lbrace
\begin{aligned}
&\frac{\partial}{\partial t} M_t = -2M_t\cdot \frac{\partial}{\partial z}M_t(z)\\
&M_0(z) = \int_\R \frac{2}{z-u}\,d\meu_0(u)
\end{aligned}
\right.\quad,
\end{equation}
and the the limit of $g_{N_k,t}$ satisfies 
\begin{equation}\label{Brezel} \frac{d}{dt}g_t = M_t(g_t), \quad g_0(z)=z. \end{equation}

If we put $f_t=g_t^{-1},$ we obtain the Loewner PDE mentioned in Section~\ref{introduction}:
$$\frac{\partial f_t(z)}{\partial t} = - \frac{\partial f_t(z)}{\partial z}\cdot M_t(z).$$

\begin{theorem}\label{LadyDi} Under the assumptions of Theorem \ref{thm:1} with $\lambda_{N,k}=\frac1{N},$ the sequences $\meu_{N,t}$ and $g_{N,t}$ converge in distribution as $N\to\infty.$
\end{theorem}
As already mentioned, this follows as soon as we know that equation \eqref{burgers} has a unique solution, which is shown, e.g., in \cite[Section 4]{MR1217451} or \cite[Section 5]{MR1440140}. We give here another short proof.
\begin{proof}
Let $M_t$ be a solution of \eqref{burgers}. As $M_t$ has no zeros in $\Ha$ we can consider $F_t:=1/M_t$ which satisfies $\frac{\partial}{\partial t} F_t = -2F_t^{-1}\cdot \frac{\partial}{\partial z}F_t$. Next we use the fact that every $F_t$ is univalent in a region 
$\Gamma_{\alpha(t), \beta(t)}$, where  $$\Gamma_{\alpha,\beta}:=\{z\in\Ha \,|\, \Im(z)>\beta, \Im(z)>\alpha|\Re(z)|\}, \alpha, \beta>0,$$ see \cite[Proposition 5.4]{MR1254116}. For $t\in[0,T]$ we find $\alpha_0$ and $\beta_0$ such that $F_t$ is univalent in $\Gamma_{\alpha_0, \beta_0}$ for all $t\in[0,T].$\\
Thus we can define $V_t(z)$ as $F_t^{-1}(z)$ for $z\in\Gamma_{\alpha_0, \beta_0}$ and a simple calculation gives
$$\frac{\partial}{\partial t} V_t(z) = \frac{2}{z}, \quad V_0(z) = (1/M_0)^{-1}(z).$$
Obviously, $V_t$, and hence also $M_t,$ is uniquely determined. 
\end{proof}

\begin{remark}\label{remus} Transforms like $\mu_t\mapsto V_t(z)$ appear in free probability theory, which was introduced by D. Voiculescu in the 1980's (in \cite[p. 3059]{MR2506464}, $V_t(z)-z$ is called Voiculescu transform). We notice that Wigner's semicircle law appears here as follows: for $\meu_0=\delta_0,$ the solution of \eqref{burgers} is given by $M_t(z) = \frac{4}{z+\sqrt{z^2-16t}},$ which is $2$ times the Cauchy transform of the centred semicircle law with variance $4t.$ \\
For relations between the chordal (and radial) Loewner equation to non-commutative probability theory, we refer to \cite{MR2053711, Loewner_monotone}.
\end{remark}
\begin{remark}\label{Tatort} In \cite{MS15}, the authors prove some geometric properties of the solution $g_t$ of \eqref{Brezel}, under the assumption  that the support of $\meu_0$ is bounded. We mention one property of this case, which will be needed later on.\\ The measures  $\meu_t$ ``grow'' continuously in the following sense: $\supp \meu_s \subset \supp \meu_t$ for all $s\leq t$ and for each $x\in \R \setminus \supp \meu_s$ there exists $T>s$ such that $x\not\in  \supp \meu_t$ for all $t\leq T.$ 
This is actually a consequence of the theory of the real Burgers equation (see \cite[Section 3.4]{MS15}).
\end{remark}

\begin{remark} Let $M_t$ be a solution of  \eqref{burgers} and $c>0$. Define $G_t(z):=c \cdot M_{c^2 t}(c\cdot z)$. Then $G_t$ also satisfies \eqref{burgers} with initial value $G_0(z)=c \cdot M_{0}(c\cdot z)$. Fix some $T>0$. As $G_0(z)\to\frac2{z}$ when $c\to\infty$, we obtain together with Remark \ref{remus} the long time behaviour 
$$ \lim_{c\to\infty} c \cdot M_{c^2 T}(c\cdot z) = \frac{4}{z+\sqrt{z^2-16 T}} \quad \text{or} \quad
M_t(z) \sim \frac{4}{z+\sqrt{z^2-16t}} \; \text{as $t\to\infty.$}$$ 
\end{remark}

\subsection{Examples}

In the following we consider three examples. In all three cases we assume that $\kappa=0$, i.e. we look at the deterministic case to make the differential equations somewhat simpler.\\ 

The proof of Theorem \ref{thm:1} shows that the sequence $\frac{d}{dt} \left(\int_\R f(x) d\til_{N,t}(x)\right)$, as a sequence of functions on $[0,T],$ is uniformly bounded. In general, this is not true for $\meu_{N,t}.$
\begin{example}\label{PrinceCharles}Let $S_N=1+\frac{N+1}{2N}$. We choose
$$x_{N,k}=\frac{k}{N^2} \text{\quad and \quad} \lambda_{N,k} = \frac{1}{S_N}\left(1 + \frac{k}{N}\right) \cdot \frac{1}{N}.$$
Obviously, $$\meu_{N,0}, \til_{N,0} \overset{{\mathbf w}}{\longrightarrow} \delta_0,$$
and $\lambda_{N,k}\leq C/N$ for some $C>0$ as $\lambda_{N,k}\leq \lambda_{N,N}\sim \frac{2}{3N}$ as $N\to\infty.$ Furthermore, as $x_{N,k}\in[0,1]$ for all $k,N,$ it is easy to see that condition \eqref{cond1} is satisfied. \\ Finally, $L_N(k/N)$ is given by $L_N(k/N)=\sum_{j=1}^k \lambda_{N,j}=\frac1{S_N}(\frac{k}{N} + \frac{k}{N}\cdot \frac{k+1}{2N}),$ which shows that $L_N$ converges uniformly to $L(x)=\frac{2}{3}(x+\frac{x^2}{2}).$ Consequently, all the assumptions of Theorem \ref{thm:1} are satisfied.
\end{example}

\begin{proposition} Under the assumptions of Example \ref{PrinceCharles}, there exists $f\in C^2_b(\R,\C)$ such that \\$\frac{d}{dt} \left(\int_\R f(x) d\meu_{N,t}(x)\right)|_{t=0}$ is unbounded.
\end{proposition}
\begin{proof}
Note that \begin{equation}\label{Karl}
\frac{\lambda_{N,k}-\lambda_{N,j}}{x_{N,k}-x_{N,j}}=1/S_N  \frac{k/N^2 - j/N^2}{k/N^2-j/N^2} = 1/S_N.\tag{$*$}
\end{equation}
Let $f\in C^2_b(\R, \C)$. Then we obtain

\begin{eqnarray*}
 && \frac{d}{dt}\left(\int_\R f(x) \, d\meu_{N,t}(x)\right) = \frac{d}{dt}\left(\sum_{k=1}^N \lambda_{N,k} f(V_{N,k}(t)) \right) 
= \sum_{k=1}^N  \lambda_{N,k} f'(V_{N,k}(t) \sum_{j\not=k}\frac{2(\lambda_{N,k}+\lambda_{N,j})}{V_{N,k}(t)-V_{N,j}(t)} \\
&=& \sum_{k=1}^N \lambda_{N,k} f'(V_{N,k}(t))\sum_{j\not=k}\frac{2\lambda_{N,j}}{V_{N,k}(t)-V_{N,j}(t)}   + 
\sum_{k=1}^N  \lambda_{N,k}^2 f'(V_{N,k}(t))\sum_{j\not=k}\frac{2}{V_{N,k}(t)-V_{N,j}(t)}  \\
&=& \int_{}\int_{x\not=y} \frac{2f'(x)}{x-y} \,
d\meu_{N,t}(x) d\meu_{N,t}(y) + 2\sum_{j\not=k}   \frac{\lambda_{N,k}^2f'(V_{N,k}(t))}{V_{N,k}(t)-V_{N,j}(t)} \\
&=&\int_{}\int_{\R^2} \frac{f'(x)-f'(y)}{x-y} \, d\meu_{N,t}(x)d\meu_{N,t}(y) +  \sum_{k=1}^N\lambda_{N,k}^2f''(V_{N,k}(t)) \\
&+& \underbrace{\sum_{j\not=k}\frac{\lambda_{N,k}^2 f'(V_{N,k}(t))- \lambda_{N,j}^2 f'(V_{N,j}(t))}{V_{N,k}(t)-V_{N,j}(t)}}_{:=T_N(t)}.
\end{eqnarray*} 
 Now assume that $f'(x)=1$ for all $x\in[0,1].$ It is easy to see that the first two terms are uniformly bounded. However,
\begin{eqnarray*}
T_N(0) &=& \sum_{j\not=k}\frac{\lambda_{N,k}^2- \lambda_{N,j}^2}{x_{N,k}-x_{N,j}} \underset{\eqref{Karl}}{=} \sum_{j\not=k}(\lambda_{N,k}+\lambda_{N,j})/S_N \geq
\sum_{j\not=k} \lambda_{N,1} / S_N \\
&=& \frac{(N^2-N)(\frac1{N} + \frac1{N^2})}{S_N^2} \to \infty \quad \text{as $N\to\infty$}.
\end{eqnarray*}
\end{proof}

Next we have a look at two examples where condition \eqref{max} is not satisfied, as for every $N$ there is one coefficient $\lambda_{N,k}=\frac1{2}.$ \\

\begin{example}\label{Johnny} For $N\geq 2$, let 
$$\text{$x_{N,k}=\frac{k}{N}$ and $\lambda_{N,k}=\frac{1}{2(N-1)}$ for all $k\not= N,$ and $x_{N,N}=2$, $\lambda_{N,N}=\frac1{2}$}.$$
\end{example}

\begin{proposition}\label{not_tight}
Let $T>0.$ Under the assumptions of Example \ref{Johnny}, the sequence $\{\meu_{N,t}\}_N$ is not tight with respect to the topology of $\mathcal{M}(T).$
\end{proposition}
\begin{proof}
We show that $V_{N,N}(t)\to +\infty$ as $N\to\infty$ for every $t>0.$ As $V_{N,N}$ carries the mass $1/2$, this proves that  $\{\meu_{N,t}\}_N$ is not tight.\\
First, we need an upper bound for $V_{N,N-1}$. For $k\in{1,...,N-1}$ we have 
\begin{eqnarray*}
dV_{N,k}(t) &\leq& \sum_{j\not=k,N}\frac{\frac{2}{N-1}}{V_{N,k}(t)-V_{N,j}(t)}dt. 
\end{eqnarray*}
 Let $W_{N,1},... W_{N,N-1}$ be the of solution the system
\begin{eqnarray*}
dW_{N,k}(t) &=&  \sum_{j\not=k,N} \frac{\frac{2}{N-1}}{W_{N,k}(t)-W_{N,j}(t)}dt, \quad W_{N,k}(0)=x_{N,k}.
\end{eqnarray*}
As the function $$(x_1,...,x_{N-1})\mapsto \left(\sum_{j\not=1,N} \frac {\frac{2}{N-1}}{x_1-x_j},..., \sum_{j\not=N-1,N} \frac {\frac{2}{N-1}}{x_{N-1}-x_j}\right)$$ is quasimonotone, it follows that $V_{N,k}(t)\leq W_{N,k}(t)$ for all $t\geq 0$ (Theorem 4.2 in \cite{MR591529}).
Note that $W_{N,1},... W_{N,N-1}$ is a simultaneous multiple SLE process for $N-1$ curves, each growing with ``speed'' $\frac{1}{2(N-1)}.$  From Remark \ref{Tatort} we conclude that there exists $T_0>0$ and a bound $B_1\in(1,2)$ such that 
$W_{N,N-1}(t)\leq B_1$ for all $t\in[0,T_0]$ and $N\geq 2.$ Hence, $V_{N,k}(t)\leq B_1 < 2$ for all $t\in[0,T_0].$ \\
This upper bound now gives us also a lower bound as follows: \\
As $\frac{d}{dt}V_{N,N}(t)\geq0$ and $V_{N,N}(0)=2$ we have $V_{N,N}(t)\geq 2$ for all $t\geq0$. Thus, for $k\in\{1,...,N-1\}$ we have
\begin{eqnarray*}
dV_{N,k}(t) &=& \sum_{j\not=k,N}\frac{\frac{2}{N-1}}{V_{N,k}(t)-V_{N,j}(t)}dt + \frac{2(1/2+\frac1{2(N-1)})}{V_{N,k}(t)-V_{N,N}(t)}dt \\
&\geq& \sum_{j\not=k,N}\frac{\frac{2}{N-1}}{V_{N,k}(t)-V_{N,j}(t)}dt + \frac{2}{B_1-2}dt.
\end{eqnarray*}
 Let $Y_{N,1},... Y_{N,N-1}$ be the of solution the system
\begin{eqnarray*}
dY_{N,k}(t) &=& \sum_{j\not=k,N}\frac{\frac{2}{N-1}}{Y_{N,k}(t)-Y_{N,j}(t)}dt + \frac{2}{B_1-2}dt, \quad Y_{N,k}(0)=x_{N,k}.
\end{eqnarray*}
From \cite{MR1440140}, Theorem 5.1, it follows that the sequence $w_{N,t}=\sum_{k=1}^{N-1}\frac1{N-1}\delta_{Y_{N,k}(t)}$ of measure-valued processes converges as $N\to\infty.$ This does not imply that $Y_{N,1}(t)$ is bounded from below, but we can conclude that, for example, $Y_{N,\lfloor N/2 \rfloor}$ is bounded from below, i.e. there exists $B_2<1$ such that $Y_{N,\lfloor N/2 \rfloor}(t) \geq B_2$ for all $t\in[0,T_0].$\\
Now we look at $V_{N,N}$, which satisfies
$$dV_{N,N} =  \sum_{j\not=N}\frac{2(\frac1{2}+\frac{1}{2(N-1)})}{V_{N,N}(t)-V_{N,j}(t)}dt \geq
 \sum_{j\leq\lfloor N/2 \rfloor}\frac{1}{V_{N,N}(t)-B_2}dt  =  \frac{\lfloor N/2 \rfloor}{V_{N,N}(t)-B_2}dt $$
for $t\in[0,T_0],$ which implies $$V_{N,N}(t)\geq B_2 + \sqrt{4 - 4 B_2 + B_2^2 - 2 t + 2 \lfloor N/2 \rfloor t}.$$
Hence, $V_{N,N}(t)\to \infty$ for every $t\in(0,T_0]$ as $N\to\infty.$ As $t\mapsto V_{N,N}(t)$ is increasing, we conclude that $V_{N,N}(t)\to \infty$ for every $t>0$.
\end{proof}
Even though $\{\meu_{N,t}\}_N$ is not tight in this example, it is easy to see that $g_{N,t}$ converges as $N\to\infty.$ If we decompose $\meu_{N,t} = \beta_{N,t} + \frac1{2}\delta_{V_{N,N}(t)},$ then it can easily be shown that $\beta_{N,t}$ converges to a process $\beta_t$ and that $P_t(z)=\int_\R\frac2{z-u} d\beta_t(u)$ satisfies a Burgers equation.

\begin{example}\label{Molly} Assume that $N=2K+1,$ $K\in\N,$ and let
$$\text{$x_{N,k}\in[-2,-1]$ and  $x_{N,2K+2-k}=-x_{N,k}\in[1,2]$ for all $k \leq K.$}$$
Assume that $x_{N,K+1}=0.$ The coefficients $\lambda_{N,k}$ are chosen as 
$$ \lambda_{N,K+1} = 1/2, \quad \lambda_{N,k}=\frac1{4K}, k\not=K+1. $$
\end{example}

As $N\to\infty$, the sequence $L_N$ converges pointwise, but not uniformly, to $L(x)=1/2 x,$ $x\in[0,1/2),$ $L(x)=1/2x + 1/2,$ $x\in[1/2,1].$

\begin{proposition} Under the assumptions of example \ref{Molly}, there exists $T_0>0$ such that the sequence $\{\meu_{N,t}\}_N$ is tight with respect to the topology of $\mathcal{M}(T_0).$
\end{proposition}
\begin{proof}
By symmetry, we have $V_{N,K+1}(t)=0$ for every $K\in\N$ and $t\geq 0$ and we can decompose the measure $\meu_{N,t}$ as $\meu_{N,t} = \beta_{N,t} + \frac1{2}\delta_{0} + \gamma_{N,t},$ where the support of $\beta_{N,t}$ is contained in $(-\infty, 0)$ and $\gamma_{N,t}(A)=\beta_{N,t}(-A)$ for every Borel set $A.$\\
Just as in the proof of Proposition \ref{not_tight}, we obtain that there exist $T_0>0$ and $B \in (-1,0)$ such that 
\begin{equation}\label{est}
\text{$V_{N,K}(t) \leq B$ for all $K\in\N$ and $t\in [0,T_0].$}
\end{equation}
 
Now let $f\in C^2_b(\R,\C).$ Then we have 

\begin{eqnarray*}
 && \frac{d}{dt}\left(\int_\R f(x) \, d\beta_{N,t}(x)\right) = \frac{d}{dt}\left(\sum_{k=1}^K \frac{1}{4K} f(V_{N,k}(t)) \right) \\
&=& \sum_{k=1}^K  \frac{f'(V_{N,k}(t))}{4K}\left(\sum_{\stackrel{j\leq K}{j\not=k}}\frac{2(\lambda_{N,k}+\lambda_{N,j})}{V_{N,k}(t)-V_{N,j}(t)} 
+ \sum_{\stackrel{j\geq K+2}{j\not=k}}\frac{2(\lambda_{N,k}+\lambda_{N,j})}{V_{N,k}(t)-V_{N,j}(t)}
 + \frac{2(\lambda_{N,k} +\lambda_{N,K+1})}{V_{N,k}(t)-V_{N,K+1}(t)}\right) \\
&=& \sum_{k=1}^K  \frac{f'(V_{N,k}(t))}{4K}\left(\sum_{\stackrel{j\leq K}{j\not=k}}\frac{\frac1{K}}{V_{N,k}(t)-V_{N,j}(t)} + \sum_{\stackrel{j\leq K}{j\not=k}}\frac{\frac1{K}}{V_{N,k}(t)+V_{N,j}(t)} + \frac{\frac1{2K}+1}{V_{N,k}(t)}\right) \\
&=& 4\int_{}\int_{x\not=y} \frac{f'(x)}{x-y} + \frac{f'(x)}{x+y} \,
d\beta_{N,t}(x) d\beta_{N,t}(y) \,  + \sum_{k=1}^K  \frac{f'(V_{N,k}(t))(\frac1{2K}+1)}{4K V_{N,k}(t)}\\
&=& 2\int_{}\int_{x\not=y} \frac{f'(x)-f'(y)}{x-y} + \frac{f'(x)+f'(y)}{x+y} \,
d\beta_{N,t}(x) d\beta_{N,t}(y) \,  + \sum_{k=1}^K  \frac{f'(V_{N,k}(t))(\frac1{2K}+1)}{4K V_{N,k}(t)}.
\end{eqnarray*} 
The term $\frac{f'(x)-f'(y)}{x-y}$ is bounded as $f\in C^2_b(\R,\C),$ and the integral over $\frac{f'(x)+f'(y)}{x+y}$ is bounded for $t\in[0,T_0]$ because of \eqref{est}. The sum is also bounded for all $t\in[0,T_0]$ because of \eqref{est}. We conclude that $\{\beta_{N,t}\}_N$ and thus $\{\meu_{N,t}\}_N$ is tight w.r.t. $\mathcal{M}(T_0).$
\end{proof}

It seems that the last two example behave in the same way when $\kappa>0.$ Figures \ref{fig2} and \ref{fig3} show simulations of the driving functions $V_{N,1},...,V_{N,N}$ for these two cases on the time interval $[0,1]$ for $N=51$ and $\kappa=1.$ The driving function with mass $\frac1{2}$ is coloured red.

\begin{figure}[ht]
\begin{minipage}[hbt]{7.5cm}
\centering
\includegraphics[width=6.5cm]{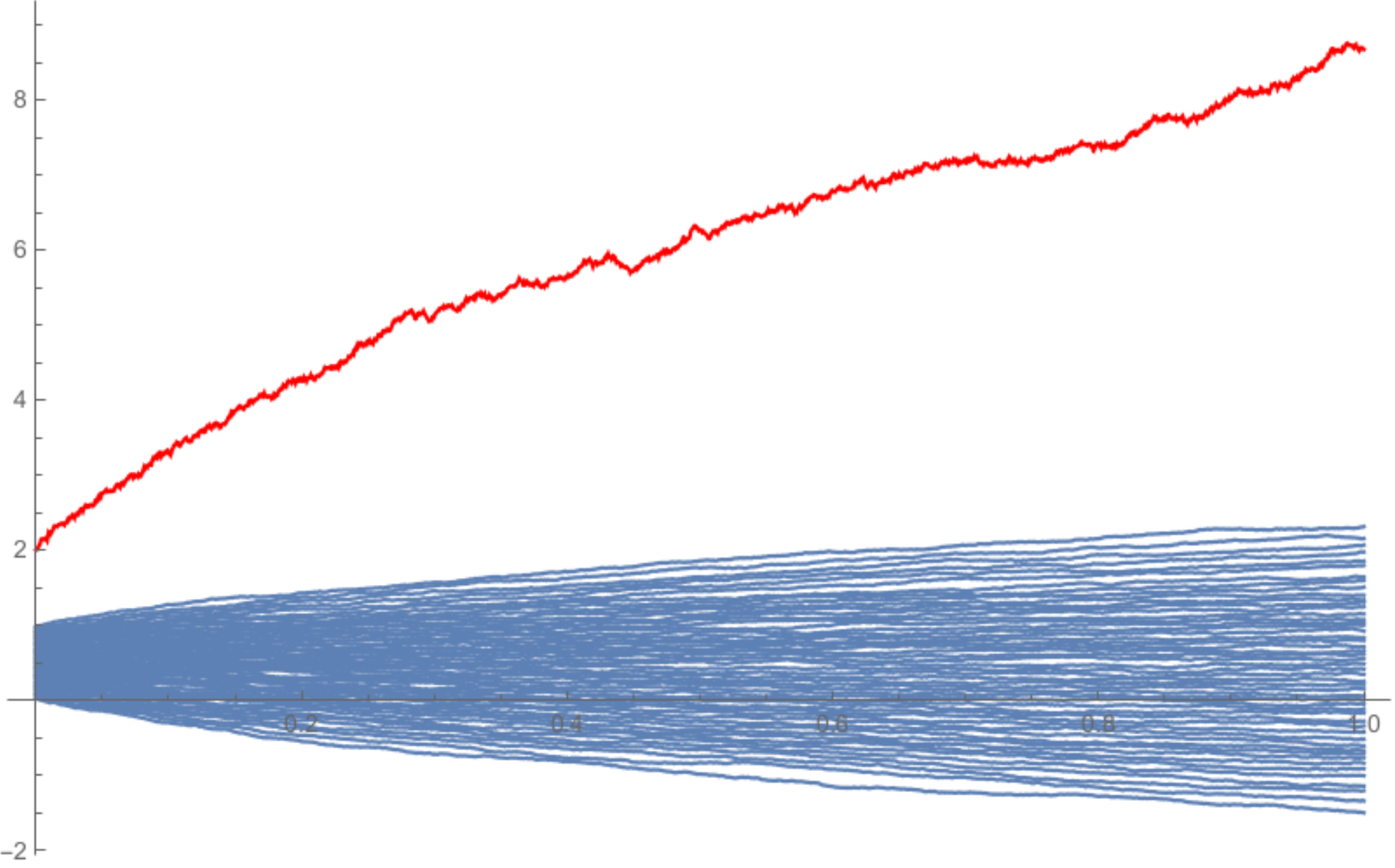}
\caption{Mass $\frac1{2}$ in $x_{N,N}$.}
\label{fig2}
\end{minipage}
\hfill
\begin{minipage}[hbt]{7.5cm}
\centering
\includegraphics[width=6.5cm]{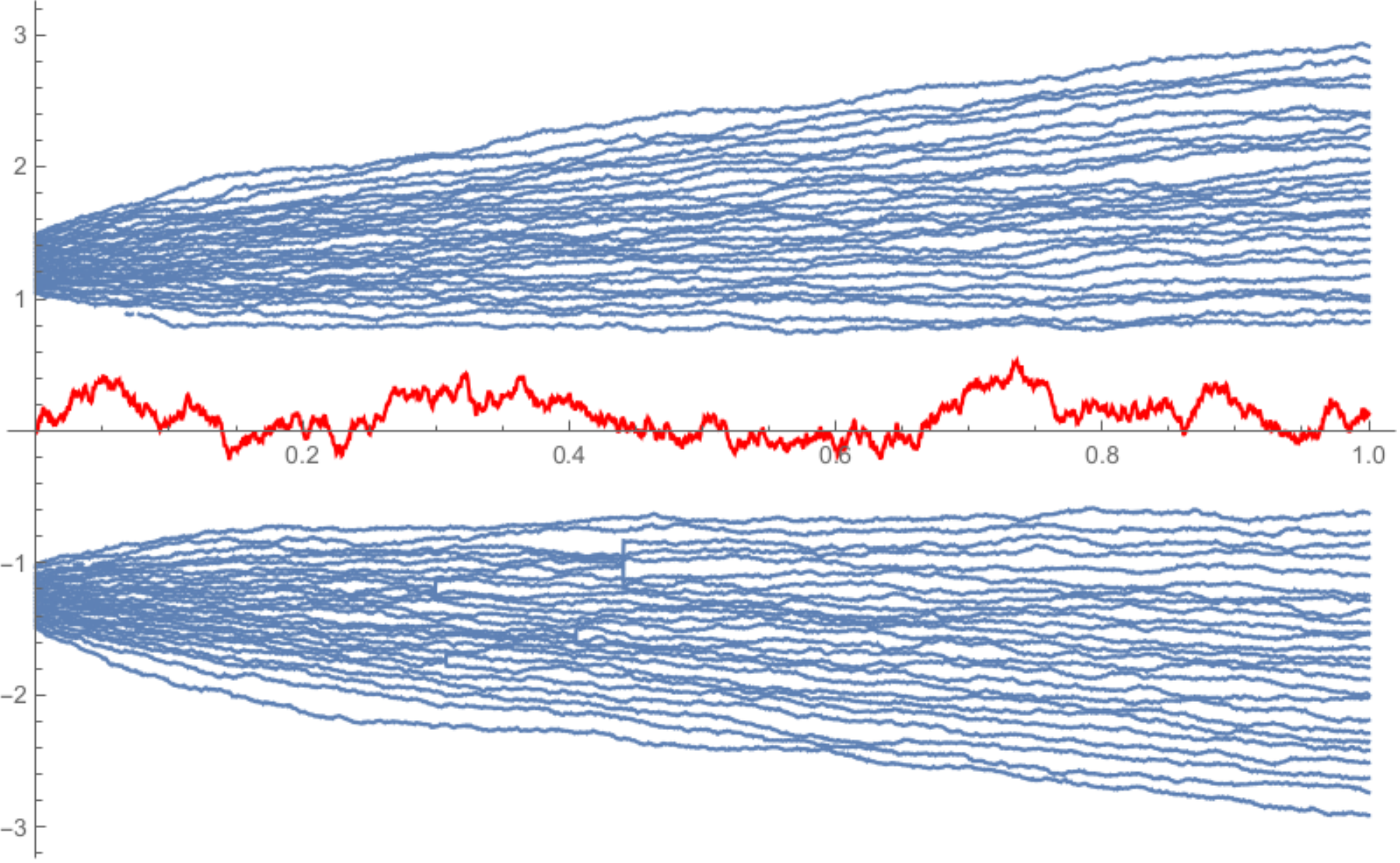}
\caption{Mass $\frac1{2}$ in $x_{N,(N-1)/2}$.}
\label{fig3}
\end{minipage}
\end{figure}

\subsection{Problems and Remarks}

\begin{enumerate}
\item  As already mentioned in Remark \ref{Napoli}, the convergence of $g_{N,t}$ from Corollary \ref{thm:2} follows as soon as we know that equation \eqref{abstractBurgers} has only one solution.
\item Example \ref{PrinceCharles} suggests that the process $\meu_t$ might not in general be differentiable (in the distributional sense) at $t=0$.\\
 \textbf{Question:} Is it always differentiable for $t>0$?\\
 Also, we notice that in \cite{MR1698948} it is shown that, for a special case, $\meu_t$ has a density with respect to the Lebesgue measure for $t>0$.\\
 \textbf{Question:} Is this always true for $\meu_t$ under the assumptions made in Theorem \ref{thm:1}?
\item Fix a parameter $\kappa\in(0,4].$ For each $N\in\N,$ we consider $2N$ boundary points $0< p_{N,1} < p_{N,2} < ... <p_{N,2N}=1$  for multiple SLE on $\Ha$.
We set $\mathbf{p}_N:=(p_{N,1},...,p_{N,2N})$. Recall that $S(\mathbf{p}_N)$ is the set of all $C_N$ configurations for these points, endowed with the probabilities given by formula \eqref{QueenMom}.\\
 
Now we can ask for the limit of $S(\mathbf{p}_N)$ as $N\to\infty$ by using an idea from combinatorics, to encode configurations into Dyck paths.\\

An $N$--Dyck path is a continuous function $d:[0,2N]\to [0,\infty)$ defined as follows:
\begin{itemize}
\item $d(0) = 0$ and $d(2N)=0,$
\item $d(k) - d(k+1)\in\{-1,+1\}$ for all $k\in\{0,...,2N-1\},$
\item for all other points $x\in [0,2N]\setminus \{0,1, ..., 2N\}$, $d(x)$ is defined by linear interpolation.
\end{itemize}

The set of all $N$--Dyck paths corresponds to the set $S(\mathbf{p}_N)$ in the following way.\\
An $N$--Dyck path can be completely described by $2N$ numbers $L_1,...,L_{2N}\in\{-1,+1\}$ representing the slopes of the $2N$ line segments. These numbers are determined by a configuration for $\mathbf{p}_N$ as follows (see the figures below for an example):
\begin{itemize}
\item[(i)] $L_{k}=+1$ and $L_{k+1}=-1$ if and only if $p_k$ and $p_{k+1}$ are connected by a simple curve;
\item[(ii)] $L_{k}=L_{k+1}=+1$ if and only if the curve connecting $p_{k+1}$ is ``contained'' in the curve connecting $p_{k}$;
\item[(iii)] $L_{k}=L_{k+1}=-1$ if and only if the curve connecting $p_{k}$ is ``contained'' in the curve connecting $p_{k+1}$.
\end{itemize}

\begin{figure}[ht]
\rule{0pt}{0pt}
\hspace{1.2cm}
\begin{minipage}[hbt]{6.5cm}
\centering
\includegraphics[width=6.5cm]{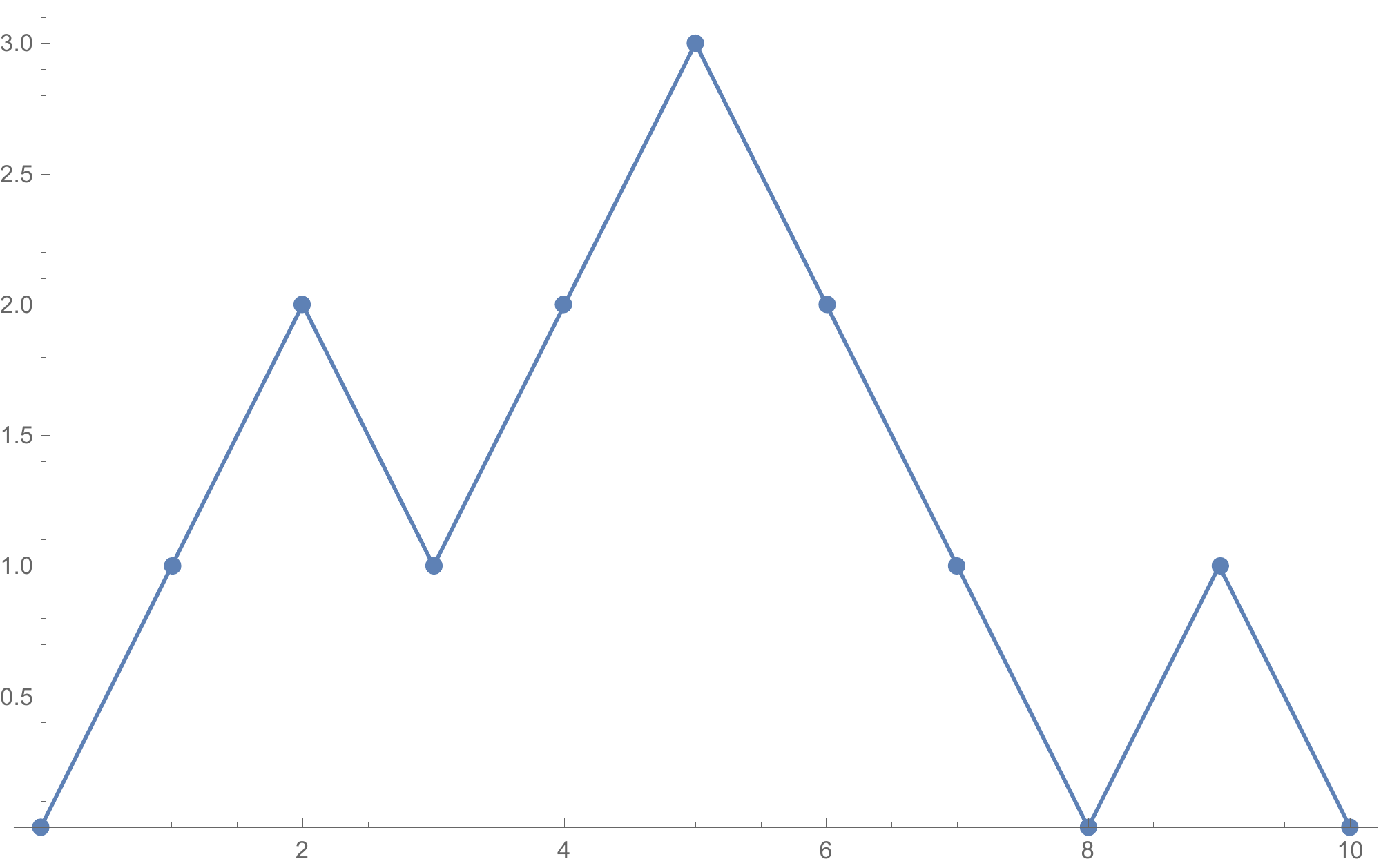}
\caption{A Dyck path for $N=5$.}
\label{fig5}
\end{minipage}
\hspace{1.2cm}
\begin{minipage}[hbt]{6.5cm}
\centering
\includegraphics[width=6.5cm]{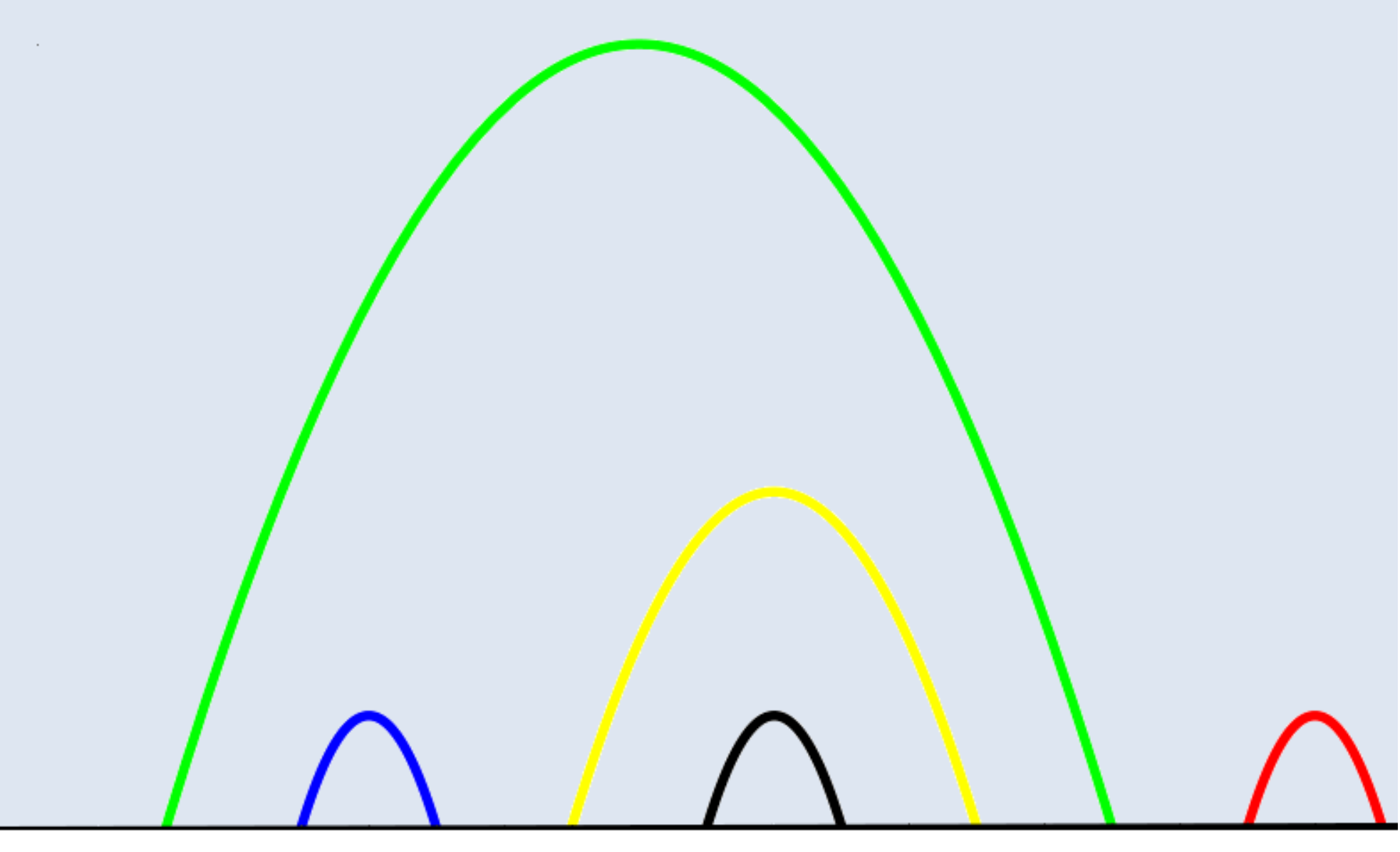}
\caption{The configuration corresponding to Figure 1.}
\label{fig6}
\end{minipage}
\end{figure}

Define also $p_{N,0}:=0$ and fix some $\gamma\in(0,1]$. Normalize now such a Dyck path $d$ to define a  normalized Dyck path as a continuous function $e_N:[0,1]\to [0,\infty)$ with $e_N(p_{N,0})=0$ and
\begin{equation}
e_N(t) = e_N(p_{N,k}) + t\cdot \frac{d(k+1)-d(k)}{(p_{{N,k+1}}-p_{N,k})^{\gamma}}
\end{equation}
for $t\in[p_{N,k}, p_{N,k+1}]$, $k=0,...,2N-1$. Then the set of all  normalized Dyck paths is a subset of the space $C([0,1], \R)$ endowed with the topology of uniform convergence. It becomes a probability space by taking the corresponding probabilities from the set $S(\mathbf{p}_N)$.
Let $E_N(t)$ be a random path from this set.\\ \textbf{Question:} Does $E_N(t)$  converge in distribution as $N\to\infty$?

\begin{remark} Take $p_k=\frac{k}{2N}$ and $\gamma=\frac1{2}.$
If all the probabilities are equally distributed, i.e. the probability for each normalized Dyck path is $\frac{1}{C_N},$ then the corresponding random path $E_N(t)$ converges in distribution to a Brownian excursion process of duration 1 (see \cite[Section 1.2]{MR2475013} and \cite{MR1989446}).
\end{remark}

Furthermore, we note that the probabilities for configurations are also considered for $\kappa>4,$ e.g. in\cite{puregeo}.\\
\textbf{Question:} What can be said about the limit of the probabilities for the set $S(\mathbf{p}_N)$ as $\kappa\to 0$?

\item The above questions can be extended to different settings like radial multiple SLE or multiple SLE in multiply connected domains (refer to \cite{Gras}). For instance, in \cite{MR2004294}, the author describes the Loewner equation for radial SLE where $N$ simple curves grow from the boundary of the unit disc $\D$ within $\D$ towards the interior point $0$. The radial analogue of Theorem \ref{LadyDi}, i.e. the coefficients in the Loewner equation are $\frac1{N}$, can be obtained simply by using the main result of \cite{MR1875671}.  
\end{enumerate}

\section{Trajectories of a certain quadratic differential}\label{Confucius}

Finally, we take a look at a Loewner equation that describes the growth of $N$ trajectories of a certain quadratic differential. By using the methods from the previous section, we obtain again an abstract differential equation for the limit case $N\to\infty$, which reduces to the Burgers equation in a special case.\\

Consider again $N$ points $x_{N,1}<x_{N,2}<...<x_{N,N}$ on $\R$ and the quadratic differential 

$$ Q_N(z)\,dz^2=\prod_{k=1}^N(z-x_{N,k})^2\prod_{j=1}^{M_N}(z-s_{N,j})^{\alpha_{N,j}}\prod_{j=1}^{M_N}(z-\overline{s_{N,j}})^{\alpha_{N,j}}\,dz^2, $$
where $M_N\in\N\cup\{0\}$, $s_{N,j}\in\Ha$ and $\alpha_{N,j}\in \Z$.\\

Then, $\R$ is a trajectory of $Q(z) dz^2$ and, for every $k$, there is exactly one trajectory starting from $x_{N,k}$ and going into the upper half-plane. As $Q$ has a zero of order 2 at $z=x_{N,k},$ this trajectory will form a $90\textdegree$-angle with $\R$ (see \cite[p. 213-215]{Pom:1975}).\\

Choose $N$ coefficients $\lambda_{N,k}\in[0,1]$ such that $\sum_{k=1}^N \lambda_{N,k}=1$. The Loewner equation

\begin{equation*}
g_{N,t}(z) = \sum_{k=1}^N \frac{2\lambda_{N,k}}{g_{N,t}(z)-V_{N,k}(t)}, \quad g_0(z) = z,
\end{equation*}

generates exactly these $N$ curves provided that the driving functions $V_{N,k}(t)$ satisfy
\begin{equation*}
\left\lbrace
\begin{aligned}
& \frac{d}{dt}V_{N,k}(t) = \sum_{j\not=k} \frac{2\lambda_{N,j}}{V_{N,k}(t)-V_{N,j}(t)} + \sum_{j=1}^{M_N} \frac{\alpha_{N,j} \lambda_{N,k}}{V_{N,k}(t)-S_{N,j}(t)} +  \sum_{j=1}^{M_N} \frac{\alpha_{N,j} \lambda_{N,k}}{V_{N,k}(t)-\overline{S_{N,j}(t)}}\\
& V_{N,k}(0) = x_{N,k}
\end{aligned}
\right.
\end{equation*}
where $S_{N,j}(t) = g_{N,t}(s_{N,j}).$

\begin{remark}
This follows from \cite[Theorem 5.1]{MR2561254}, where all the degrees $\mu_k^\pm$ are equal to $0$ in our case, as the trajectories form a $90\textdegree$-angle with the real axis, which is also a trajectory of $Q(z)dz^2$ (check p. 564 in \cite{MR2561254}).
\end{remark}

Next, define the probability measure $\mu_{N,t}=\sum_{k=1}^N \lambda_{N,k}\delta_{V_{N,k}(t)}$.

\begin{remark} If $M_N=0$ for all $N\in\N$ and $\lambda_{N,k}\leq C/N$ for all $k, N$ and some $C>0$, then, by the proof of Theorem \ref{thm:1}, it is easy to that the following holds: \\
If $\mu_{N,0}\to \mu$ as $N\to\infty$ such that \eqref{cond1} is satisfied, then the limit $\mu_t$ of $\mu_{N,t}$ exists, and the  transform $M_{t}(z)=\int_\R \frac{2}{z-u}\,d\mu_t(u)$ satisfies the Burgers equation
$$\frac{\partial}{\partial t} M_t = -M_t\cdot \frac{\partial}{\partial z}M_t(z),\quad
M_0(z) = \int_\R \frac{2}{z-u}\,d\mu(u).$$
Note that this is equation \eqref{burgers} with the $2$ replaced by $1$. The limit $g_t$ of $g_{N,t}$ satisfies $\frac{\partial}{\partial t}g_t = M_t(g_t)$ and a simple calculation shows that $\frac{\partial}{\partial t}M_t(g_t(z)) = 0,$ which implies that
$t\mapsto g_t(z_0)$, for $z_0\in\Ha$ fixed, describes a straight line, and that $M_t(g_t(z)) = M_0(z)$.
\end{remark}

Assume now $\lambda_{N,k}=\frac1{N}$ for all $k$ and $N$. \\
First, we introduce a second measure-valued process 
$$\sigma_{N,t} = \sum_{j=1}^{M_N} \frac{\alpha_{N,j}}{N}\delta_{S_{N,j}(t)},$$

and we assume that there exists a compact set $K\subset \Ha$ such that 

\begin{equation}\label{sigmacond}
\supp \sigma_{N,0} \subset K \quad \text{for all $N\in\N$.} \tag{d}
\end{equation}

\begin{theorem} Let $\lambda_{N,k}=\frac1{N}$ for all $k$ and $N.$ Assume that $\mu_{N,0}$ converges weakly to the probability measure $\mu$ such that \eqref{cond1} holds. Furthermore, assume that \eqref{sigmacond} holds and that $\sigma_{N,0}$ converges weakly to a finite (signed)
measure $\sigma$ as $N\to\infty.$ Then there exists $T>0$ such that $\{\mu_{N,t}\}_{N\in\N}$ is tight as a sequence in $\mathcal{M}(T).$ \\
Moreover, let $\mu_{N_k,t}$ be a converging subsequence with limit $\mu_t.$ Then the following two statements hold:
\begin{itemize}
\item[(i)]  $\sigma_{N_k,t}$ converges to a process $\sigma_t$ and 
\begin{eqnarray*}
&&\frac{d}{dt}\left(\int_\R f(x) \, d\mu_{t}(x)\right) = \int_{\R^2} \frac{f'(x)-f'(y)}{x-y} d\mu_{t}(x)d\mu_{t}(y) + 2\Re\left(\int_\Ha\int_\R  \frac{f'(x)}{x-z} d\mu_{t}(x)d\sigma_{t}(z) \right), \\
 && \frac{d}{dt}\left(\int_\Ha h(z) \, d\sigma_{t}(z)\right) = \int_\R\int_\Ha  \frac{2h'(z)}{z-y} d\sigma_{t}(z)d\mu_{t}(y),
\end{eqnarray*} 
for every $f\in C^2_b(\R,\C)$ and continuously differentiable $h:\Ha\to\C$ with $h'$ bounded.
\item[(ii)] Let  $M_t(z):=\int_\R\frac{2}{z-x}\, d\mu_t(x).$ Then $g_{N_k,t}$ converges locally uniformly to $g_t$ which satisfies the following system of (abstract) differential equations:
\begin{eqnarray*}
&&\frac{d}{dt}g_t(z) = M_t(g_t),\\
&&\frac{\partial}{\partial t}M_t(z) = -\frac{\partial}{\partial z}M_t(z)\cdot M_t(z)  +
  2\Re\left(\int_\Ha  \frac{M_t(z)}{(z - g_t(w))^2} -\frac{ M_t(g_t(w))}{(z - g_t(w))^2} - \frac{\frac{\partial}{\partial z}M_t(z)}{(z - g_t(w))} d\sigma(w)\right).
 \end{eqnarray*}
\end{itemize}
\end{theorem}
\begin{proof}
First we note that $\sigma_{N,t}$ is the pushforward of $\sigma_{N,0}$ w.r.t. $g_{N,t}$, i.e.
\begin{equation}\label{push}
\sigma_{N,t}=(g_{N,t})_*(\sigma_{N,0}).
\end{equation}
A normality argument plus assumption \eqref{sigmacond} would yield the existence of $T>0$ and a compact set $K_T\subset \Ha$ such that 
\begin{equation}\label{KingGeorge}\text{$\supp\sigma_{N,t}\subset K_T$ for all $t\in[0,T].$}
\end{equation}

Now let $f\in C^2_b(\R,\C).$ Then
\begin{eqnarray*}
 && \frac{d}{dt}\left(\int_\R f(x) \, d\mu_{N,t}(x)\right) = \frac{d}{dt}\left(\sum_{k=1}^N  \frac1N f(V_{N,k}(t)) \right)\\
 &=& \sum_{k=1}^N  \frac1N f'(V_{N,k}(t)) \cdot 
 \left(\sum_{j\not=k} \frac{2/N}{V_{N,k}(t)-V_{N,j}(t)} + 
\sum_{j=1}^{M_N} \frac{\alpha_{N,j}/N}{V_{N,k}(t)-S_{N,j}(t)} +  \sum_{j=1}^{M_N} \frac{\alpha_{N,j}/N}{V_{N,k}(t)-\overline{S_{N,j}(t)}}\right) \\
&=& \int_{x\not=y} \frac{f'(x)-f'(y)}{x-y} d\mu_{N,t}(x)d\mu_{N,t}(y) + 2\Re\left(\int_\Ha\int_\R  \frac{f'(x)}{x-z} d\mu_{N,t}(x)d\sigma_{N,t}(z) \right).
\end{eqnarray*} 
As in the proof of Theorem \ref{thm:1}, we conclude that the first term is bounded. The second one is bounded for all $t\in[0,T]$ because of \eqref{KingGeorge} and as $\sigma$ is finite. \\
Recall that $S_{N,j}(t) = g_{N,t}(s_{N,j})$. Thus, for any continuously differentiable $h:\Ha\to\C$, with $h'$ bounded, we get
\begin{eqnarray*}
 && \frac{d}{dt}\left(\int_\Ha h(z) \, d\sigma_{N,t}(z)\right) = \frac{d}{dt}\left(\sum_{j=1}^{M_N} \frac{\alpha_{N,j}}{N} h(S_{N,j}(t)) \right)\\
&=& \sum_{j=1}^{M_N}  \frac{\alpha_{N,j}}{N} h'(S_{N,j}(t)) \cdot 
\sum_{k=1}^N \frac{2/N}{S_{N,j}(t)-V_{N,k}(t)} = \int_\R\int_\Ha  \frac{2h'(z)}{z-y} d\sigma_{N,t}(z)d\mu_{N,t}(y),
\end{eqnarray*} 
which is also bounded for all $t\in[0,T]$.\\
As in the proof of  Theorem \ref{thm:1}, we conclude tightness of the sequences $\{\mu_{N,t}\}_{n\in\N}$ and $\{\sigma_{N,t}\}_{n\in\N}$. It should be noted that we do not need a condition like \eqref{cond1} for the convergence of $\sigma_{N,0}$, as we assumed that the support of $\sigma_{N,0}$ is contained in a compact set independent of $N.$\\

Now let $\mu_t$ be the limit of a converging subsequence $\mu_{N_k,t}$. Relation \eqref{push} implies that $\sigma_{N,t}$ converges to $\sigma_t:=(g_t)_*(\sigma)$ as $N\to\infty,$ and we obtain statement (i).\\
As in the proof of Corollary \ref{thm:2}, we conclude that $g_{N_k,t}$ converges locally uniformly to $g_t$ which satisfies
$$\frac{d}{d t}g_t = M_t(g_t).$$

Finally,  we can write the first equation of (i) as
 \begin{eqnarray*}
&&\frac{d}{dt}\left(\int_\R f(x) \, d\mu_{t}(x)\right) = \int_{\R^2} \frac{f'(x)-f'(y)}{x-y} d\mu_{N,t}(x)d\mu_{N,t}(y) + 2\Re\left(\int_\Ha\int_\R  \frac{f'(x)}{x-g_t(w)} d\mu_{t}(x)d\sigma(w) \right).
\end{eqnarray*} 
For $f(x)=\frac{2}{z-x},$ $z\in\Ha,$ this becomes (and we use the calculation from the proof of Corollary \ref{thm:2})
\begin{eqnarray*}
&&\frac{\partial}{\partial t}M_t(z) = -\frac{\partial}{\partial z}M_t(z)\cdot M_t(z) +\\
&&  2\Re\left(\int_\Ha\int_\R  \frac{2}{(z-x) (z - g_t(w))^2} - \frac{2}{( g_t(w)-x) (z - g_t(w))^2} + \frac{2}{(x - 
    z)^2 (z - g_t(w))} d\mu_{t}(x)d\sigma(w)\right)\\
&=& -\frac{\partial}{\partial z}M_t(z)\cdot M_t(z)  +
  2\Re\left(\int_\Ha  \frac{M_t(z)}{(z - g_t(w))^2} -\frac{ M_t(g_t(w))}{(z - g_t(w))^2} - \frac{\frac{\partial}{\partial z}M_t(z)}{(z - g_t(w))} d\sigma(w)\right)
 \end{eqnarray*}
and we are done.
\end{proof}
Figure \ref{fig7} shows a stream plot of the trajectories for $Q(z) = \prod_{k=0}^9(z-2k/9+1)^2 \cdot (z-i)^{-10}\cdot  (z+i)^{-10},$ and in Figure \ref{fig8}, $z=i$ is a zero of $Q$, i.e. $Q(z) = \prod_{k=0}^9(z-2k/9+1)^2 \cdot (z-i)^{10}\cdot  (z+i)^{10}.$
\begin{figure}[ht]
\rule{0pt}{0pt}
\hspace{1.2cm}
\begin{minipage}[hbt]{7cm}
\centering
\includegraphics[width=7cm]{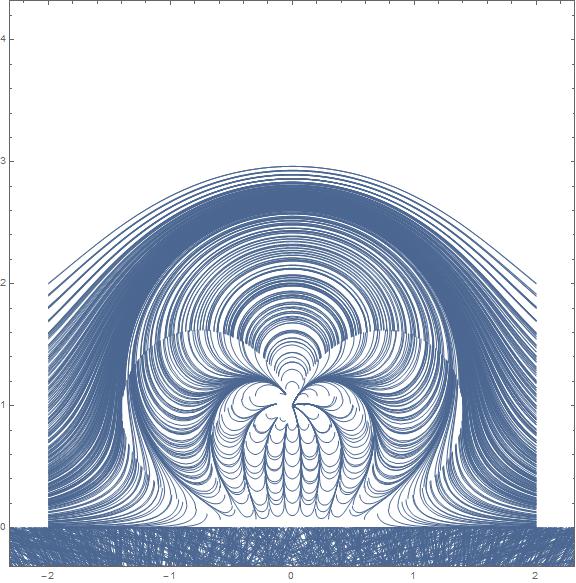}
\caption{Pole of order $N$ at $z=i.$}
\label{fig7}
\end{minipage}
\hspace{1.2cm}
\begin{minipage}[hbt]{7cm}
\centering
\includegraphics[width=7cm]{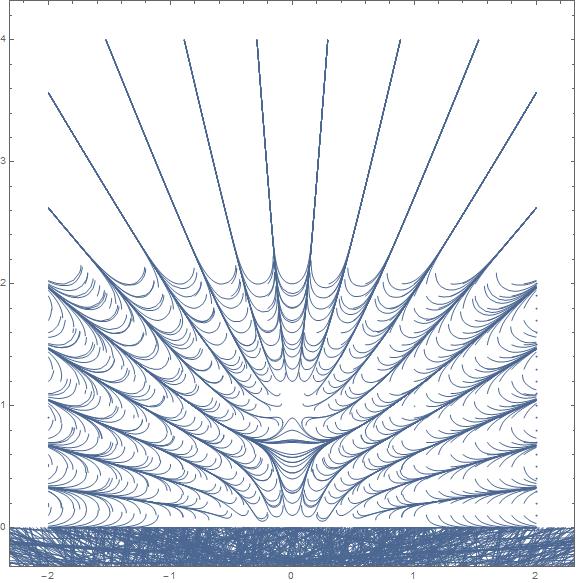}
\caption{Zero of order $N$ at $z=i.$}
\label{fig8}
\end{minipage}
\end{figure}

\def\cprime{$'$}
\providecommand{\bysame}{\leavevmode\hbox to3em{\hrulefill}\thinspace}
\providecommand{\MR}{\relax\ifhmode\unskip\space\fi MR }
\providecommand{\MRhref}[2]{%
  \href{http://www.ams.org/mathscinet-getitem?mr=#1}{#2}
}
\providecommand{\href}[2]{#2}

\vspace{1cm}
{\small
\begin{tabular}{ll}
\emph{Andrea del Monaco}: &	Universit\`{a} di Roma ``Tor Vergata'', 00133 Roma, Italy.\\[2mm]
\emph{Ikkei Hotta}: & Yamaguchi University, Ube 755-8611, Japan.\\[2mm]
\emph{Sebastian Schlei\ss inger}: & Universit\`{a} di Roma ``Tor Vergata'', 00133 Roma, Italy.
\end{tabular}

}

\end{document}